\documentclass[11pt]{article}
 \usepackage{amsmath,amssymb,amsthm}
 \RequirePackage[dvips]{graphicx}
 \def\draw #1 by #2 (#3){
  \vbox to #2{
    \hrule width #1 height 0pt depth 0pt
    \vfill
    \special{picture #3} 
    }
  }

 \def\scaleddraw #1 by #2 (#3 scaled #4){{
  \dimen0=#1 \dimen1=#2
  \divide\dimen0 by 1000 \multiply\dimen0 by #4
  \divide\dimen1 by 1000 \multiply\dimen1 by #4
  \draw \dimen0 by \dimen1 (#3 scaled #4)}
  }

\newtheorem{theorem}{Theorem}[section]
\newtheorem{example}{Example}
\newtheorem{problem}[example]{Problem}
\newtheorem{defin}[theorem]{Definition}
\newtheorem{lemma}[theorem]{Lemma}
\newtheorem{corollary}[theorem]{Corollary}
\newtheorem{conj}{Conjecture}

\newtheorem{nt}{Note}

\newenvironment{conjecture}{\begin{conj} \em}{\end{conj}}
 \newcommand{\singlespacing}{\let\CS=\@currsize\renewcommand{\baselinestretch}{1}\tiny\CS}
 \newcommand{\oneandahalfspacing}{\let\CS=\@currsize\renewcommand{\baselinestretch}{1.25}\tiny\CS}
 \newcommand{\doublespacing}{\let\CS=\@currsize\renewcommand{\baselinestretch}{1.35}\tiny\CS}

 \newtheorem{rule-def}[theorem]{Rule}

\begin{document}

\baselineskip 16pt
 \newcommand{\la}{\lambda}
 \newcommand{\si}{\sigma}
 \newcommand{\ol}{1-\lambda}
 \newcommand{\be}{\begin{equation}}
 \newcommand{\ee}{\end{equation}}
 \newcommand{\bea}{\begin{eqnarray}}
 \newcommand{\eea}{\end{eqnarray}}

 \begin{center}
 {\Large \bf Remoteness and distance eigenvalues of a graph}\\

  \vspace{8mm}

 {\large \bf Huiqiu Lin$^{a}$\,, Kinkar Ch. Das $^{b}$\,, Baoyindureng Wu$^{c}$}

 \vspace{6mm}

 \baselineskip=0.20in

 $^{a}${\it Department of Mathematics, East China University of Science and Technology, \\
 Shanghai, P. R. China\/}\\{\rm e-mail:} {\tt huiqiulin@126.com}

 $^{b}${\it Department of Mathematics, Sungkyunkwan University, \\
 Suwon 440-746, Republic of Korea\/} \\[2mm]
 {\rm e-mail:} {\tt kinkardas2003@googlemail.com}\\[2mm]
 $^c${\it College of Mathematics and System Sciences, Urumqi, Xinjiang 830046, P. R. China}\\[2mm]
 {\rm e-mail:} {\tt wubaoyin@hotmail.com}

 \end{center}

 \vspace{5mm}

 \baselineskip=0.23in

 \begin{abstract}

Let $G$ be a connected graph of order $n$ with diameter $d$.
Remoteness $\rho$ of $G$ is the maximum average distance from a vertex to all others and
$\partial_1\geq\cdots\geq \partial_n$ are the distance eigenvalues of $G$. In \cite{AH},
Aouchiche and Hansen conjectured that $\rho+\partial_3>0$ when $d\geq 3$ and $\rho+\partial_{\lfloor\frac{7d}{8}\rfloor}>0.$
In this paper, we confirm these two conjectures. Furthermore, we give lower bounds on $\partial_n+\rho$ and $\partial_1-\rho$ when
$G\ncong K_n$ and the extremal graphs are characterized.
\vspace{3mm}


\vspace{3mm}

 \end{abstract}

 \baselineskip=0.30in

 \section{Introduction}
 Throughout this paper we consider simple, undirected and connected graphs.
 Let $G$ be a graph with vertex set $V(G)=\{v_1,\,v_2,\ldots,\,v_n\}$ and edge set $E(G)$, where $|V(G)|=n$, $|E(G)|=m$.
 Also let $d_i$ be the degree of the vertex $v_i\in V(G)$. For $v_i,\,v_j\in V(G)$, the distance between vertices $v_i$ and $v_j$ is the
 length of a shortest path connecting them in $G$, denoted by $d_G(v_i,\,v_j)$ or $d_{ij}$\,. The diameter of a graph
 is the maximum distance between any two vertices of $G$. Let $d$ be the diameter of $G$. The transmission $Tr(v_i)$ $(\mbox{or }D_i\mbox{ or }D_i(G))$ of vertex $v_i$ is
 defined to be the sum of distances from $v_i$ to all other vertices, that is,
              $$Tr(v_i)=\sum\limits_{v_j\in V(G)}\,d(v_i,\,v_j).$$
The remoteness of $G$ is denoted by $$\rho=\max_{v\in V}\frac{Tr(v)}{n-1}.$$

The distance matrix of $G$, denoted by $D(G)$ or simply by $D$, is the symmetric real matrix with $(i,\,j)$-entry being $d_G(v_i,\,v_j)$ $(\mbox{or }d_{ij})$. The distance
eigenvalues (resp. distance spectrum) of $G$, denoted by
$$\partial_1(G)\geq \partial_2(G)\geq \cdots\geq \partial_n(G).$$
Aouchiche and Hansen \cite{AH} posed the following two conjectures which are related to the remoteness and distance eigenvalues.

\begin{conjecture}
Let $G$ be a graph on $n\geq 4$ vertices with diameter $d\geq 3$, remoteness $\rho$ and distance eigenvalues $\partial_1\geq \partial_2\geq\cdots\geq \partial_n.$
Then $$\rho+\partial_3>0.$$
\end{conjecture}

In this paper, we confirm the conjecture and give the lower bound on $\rho+\partial_3$ when $d$ is equal to 2.
\begin{theorem}\label{thm1}
Let $G$ be a connected graph of order $n\geq 4$ with diameter $d$, remoteness $\rho(G)$
and distance eigenvalues $\partial_1\geq\cdots\geq \partial_n$. Then we have the following statements.

\noindent (i) If $d=2$, then $$\rho(G)+\partial_3\geq \frac{\left\lceil\frac{n}{2}\right\rceil-2}{n-1}-1$$
with equality holding if and only if $G\cong K_{n_1,n_2}$.

\noindent (ii) If $d\geq 3$, then $$\rho(G)+\partial_3>0.$$
\end{theorem}

\begin{conjecture}\label{thm2}
Let $G$ be a connected graph of order $n\geq 4$ with diameter $d$, remoteness $\rho$ and distance spectrum
$\partial_1\geq \partial_2\geq\cdots\geq \partial_n$. Then $$\rho+\partial_{\lfloor\frac{7d}{8}\rfloor}>0.$$
\end{conjecture}

In the paper, we confirm the conjecture too. In Section 3, we give lower bounds on $\partial_n+\rho$ and $\partial_1-\rho$ when
$G\ncong K_n$ and the extremal graphs are characterized.

\section{Proofs}

\begin{lemma}\label{lem2.1}{\rm\cite{Lin-1}}
Let $G$ be a connected graph with diameter $d$. Then $$\partial_n\geq -d$$
with equality holding if and only if $G$ is complete multipartite graph.
\end{lemma}

\begin{lemma}\label{lem2.2}
Let $G$ be a connected graph of order $n$ with diameter $d$ and remoteness $\rho$. Then $$\rho\geq \frac{d}{2}.$$
\end{lemma}

\begin{proof}
Suppose that $P_{d+1}=v_1v_2\cdots v_{d+1}$ is a diameter path of $G$. Then for each vertex $v\in V(G)\backslash V(P_{d+1})$,
we have $$d(v,v_1)+d(v,v_{d+1})\geq d.$$
Therefore,
\begin{eqnarray*}
\rho&\geq& \max\{tr(v_1),tr(v_{d+1})\}\\[2mm]
&\geq&\frac{tr(v_1)+tr(v_{d+1})}{2}\\[2mm]
&\geq&\frac{2(1+\cdots+d)+d(n-d-1)}{2}\\[2mm]
&=&\frac{d}{2}.
\end{eqnarray*}
This completes the proof.
\end{proof}

Now we are ready to give the proof of Theorem \ref{thm1}.

\vspace{3mm}
\noindent\textbf{Proof of Theorem \ref{thm1}.} If $d\geq 3$, then $D(G)$ contains $D(P_4)$ as a principle submatrix.
Then $$\partial_3\geq \lambda_3(D(P_4))>-1.2.$$ Then by Lemma \ref{lem2.2}, we have $$\rho+\partial_3\geq \frac{d}{2}-1.2>0.$$
This complete the proof of Theorem \ref{thm1} (ii). So in the following, we may assume that $d=2.$
Note that $$\rho(K_{\lceil\frac{n}{2}\rceil,\lfloor\frac{n}{2}\rfloor})=\frac{\left\lceil\frac{n}{2}\right\rceil-2}{n-1}+1.$$
Then by Lemma \ref{lem2.1}, we have
\begin{eqnarray*}
\rho(K_{\lceil\frac{n}{2}\rceil,\lfloor\frac{n}{2}\rfloor})+\partial_3(K_{\lceil\frac{n}{2}\rceil,\lfloor\frac{n}{2}\rfloor})
&=&\frac{\left\lceil\frac{n}{2}\right\rceil-2}{n-1}-1.
\end{eqnarray*}
If $G$ is a tree, then $G\cong K_{1,n-1}$. Then by Lemma \ref{lem2.1}, we have
\begin{eqnarray*}
\rho(K_{1,n-1})+\partial_3(K_{1,n-1})&=&-2+\frac{1+2(n-2)}{n-1}\\[2mm]
&>&\frac{\left\lceil\frac{n}{2}\right\rceil-2}{n-1}-1\\[2mm]
&=&\rho(K_{\lceil\frac{n}{2}\rceil,\lfloor\frac{n}{2}\rfloor})+\partial_3(K_{\lceil\frac{n}{2}\rceil,\lfloor\frac{n}{2}\rfloor}).
\end{eqnarray*}
Let $g$ be the girth of $G.$ Since $d=2$, we only need to consider $3\leq g\leq 5.$
Note that $\partial_3(C_3)\geq-1$ and $\partial_3(C_5)>-1$. Then $\partial_3(G)\geq-1$ if either $g=3$ or $g=5$.
Then $$\rho(G)+\partial_3(G)>0>\rho(K_{\lceil\frac{n}{2}\rceil,\lfloor\frac{n}{2}\rfloor})+\partial_3(K_{\lceil\frac{n}{2}\rceil,\lfloor\frac{n}{2}\rfloor}).$$
If $g=4$, then suppose that $v_1v_2v_3v_4$ is a $C_4$ in $G$. If there exist a vertex $v\in V(G)\backslash\{v_1,v_2,v_3,v_4\}$ such that $d_{C_4}(v)=1$, then $G$
contains $C_5$ as an induced subgraph since $d=2$. Then similar to the above, we have $$\rho(G)+\partial_3(G)>\rho(K_{\lceil\frac{n}{2}\rceil,\lfloor\frac{n}{2}\rfloor})+\partial_3(K_{\lceil\frac{n}{2}\rceil,\lfloor\frac{n}{2}\rfloor}).$$
Then for each vertex in $v\in V(G)\backslash\{v_1,v_2,v_3,v_4\}$, we have $d_{C_4}(v)=2$. Then we have $G\cong K_{n_1,n_2}$. Note that $\partial_3(K_{n_1,n_2})=-2$ and
$$\rho(K_{n_1,n_2})\geq \rho(K_{\lceil\frac{n}{2}\rceil,\lfloor\frac{n}{2}\rfloor})$$ with equality holding if and only if either
$n_1=\lceil\frac{n}{2}\rceil\,,n_2=\lfloor\frac{n}{2}\rfloor$ or $n_1=\lfloor\frac{n}{2}\rfloor\,,n_2=\lceil\frac{n}{2}\rceil$.
This completes the proof of Theorem \ref{thm1} (i).  $ \ \ \ \ \ \Box$

\begin{lemma}\label{lem2.3}{\rm \cite{Merris}}
Let $G$ be a connected graph of order $n$. Let $\partial_1\geq \cdots\geq \partial_n$ be the eigenvalues of $D(G)$
and let $\lambda_1\geq \cdots\geq \lambda_{n-1}>\lambda_n=0$ be the eigenvalues of $L(G)$. Then
$$0>-\frac{2}{\lambda_1}\geq\partial_2\geq-\frac{2}{\lambda_2}\geq\cdots\geq-\frac{2}{\lambda_{n-2}}\geq\partial_{n-1}\geq-\frac{2}{\lambda_{n-1}}\geq\partial_{n}.$$
\end{lemma}

It is known that $\lambda_i(P_n)=2+2\cos\frac{i\pi}{n}$ for $i=1,\ldots,n$. Now, we are ready to present the proof of Conjecture \ref{thm2}.

\vspace{4mm}
\noindent\textbf{Proof of Conjecture \ref{thm2}.} We may assume that $P_{d+1}$ is a diameter path of $G$.  Then by Lemma \ref{lem2.3}, we have
\begin{eqnarray*}
\partial_{\lfloor\frac{7d}{8}\rfloor}(G)&\geq&\partial_{\lfloor\frac{7d}{8}\rfloor}(P_{d+1})\\[2mm]
&\geq& -\frac{2}{\lambda_{\lfloor\frac{7d}{8}\rfloor}(P_{d+1})}\\[2mm]
&=&-\frac{2}{2+2\cos\frac{\lfloor\frac{7d}{8}\rfloor\pi}{d+1}}\\
&=&-\frac{1}{2\cos^2\frac{\lfloor\frac{7d}{8}\rfloor\pi}{2(d+1)}}\\[2mm]
&=&-\frac{1}{2-2\sin^2\frac{\lfloor\frac{7d}{8}\rfloor\pi}{2(d+1)}}\\[2mm]
\end{eqnarray*}
Note that $\sin^2\frac{\frac{7d}{8}\pi}{2(d+1)}\geq \sin^2\frac{\lfloor\frac{7d}{8}\rfloor\pi}{2(d+1)}$. So we only need to show that $1-\sin^2\frac{\frac{7d}{8}\pi}{2(d+1)}>\frac{1}{d}$, it is equivalent to show that $1-\frac{1}{d}>\sin^2\frac{\frac{7d}{8}\pi}{2(d+1)}$. Let $f(x)=1-\frac{1}{x}-\sin^2\frac{7x\pi}{16(x+1)}.$
Then

$$f'(x)=1+\frac{1}{x^2}-\sin\frac{7x\pi}{16(x+1)}\frac{7\pi\times16(x+1)-7x\pi\times16}{[16(x+1)]^2}>0,$$
which implies that $f(x)$ is strictly increase function. Note that
$$f(11)=\frac{10}{11}-\sin^2\frac{77\pi}{192}>0.$$ Then $f(x)>0$ for $x\geq 11$. It follows that $1-\frac{1}{d}>\sin^2\frac{\frac{7d}{8}\pi}{2(d+1)}$ for $d\geq 11.$ Then by Lemma \ref{lem2.2}, we have $$\rho+\partial_{\lfloor\frac{7d}{8}\rfloor}>\frac{d}{2}-\frac{1}{2-2\sin^2\frac{\lfloor\frac{7d}{8}\rfloor\pi}{2(d+1)}}>\frac{d}{2}-\frac{d}{2}=0 \ \mbox{for $d\geq11$.}$$

\vspace{3mm}
\begin{center}
\begin{tabular}{|c|c|c|c|c|c|c|c|c|c|}
  \hline
 \  $d$ & 2 & 3 & 4 & 5 & 6 & 7 & 8 & 9 & 10 \\[2mm]
  \hline
\  $\frac{d}{2}$ & 1 & 1.5 & 2 & 2.5 & 3 & 3.5 & 4 & 4.5 & 5 \\[2mm]
  \hline
\  $\partial_{\lfloor\frac{7d}{8}\rfloor}$ & 2.73 &  -0.5858 & -0.7639 & -1 & -1.2862 & -1.6199 & -2 & -1.5053 & -1.7831 \\[1.2mm]
  \hline
\end{tabular}

\vspace{3mm}
Table 1. Values of $\frac{d}{2}$ and $\partial_{\lfloor\frac{7d}{8}\rfloor}$ for $d=2,\ldots,10.$

\end{center}

For $2\leq d\leq 10$, by Table 1, we have $\frac{d}{2}>-\partial_{\lfloor\frac{7d}{8}\rfloor}$,
which implies that $\frac{d}{2}+\partial_{\lfloor\frac{7d}{8}\rfloor}>0$. This completes the proof.
$ \ \ \ \ \ \Box$

\section{More results on the remoteness and distance eigenvalues}
\begin{lemma}\label{lem3.1}
Let $G$ be a connected graph of order $n$ with diameter $d$ and remoteness $\rho$. Then $$\rho\leq d-\frac{d^2-d}{2(n-1)}.$$
\end{lemma}
\begin{proof}
Let $v$ be the vertex with maximum transmission of $G.$ Then $$tr(v)\leq 1+2+\cdots+d+d(n-1-d)=dn-\frac{d^2-d}{2(n-1)}.$$
It follows that $$\rho\leq d-\frac{d^2-d}{2(n-1)}.$$ This completes the proof.
\end{proof}

Lin \cite{Lin-1} showed that $\partial_n(G)\leq -d$ with equality holding if and only if $G$ is complete multipartite graph.
Thus we have the following result.

\begin{theorem}\label{thm2.1}
Let $G$ be a connected graph of order $n$ with diameter $d$ and remoteness $\rho$. Then
$$\rho+\partial_n\leq -\frac{d^2-d}{2(n-1)}$$ with equality holding if and only if $G\cong K_n.$
\end{theorem}

\begin{corollary}{\rm \cite{AH}}
Let $G$ be a connected graph of order $n$ with diameter $d$ and remoteness $\rho$. Then
$$\rho+\partial_n\leq 0$$ with equality holding if and only if $G\cong K_n.$
\end{corollary}

Indulal \cite{Indulal} gave the following lower bound on the distance spectral radius of $G$.

\begin{lemma}{\rm\cite{Indulal}}\label{lem3.2}
Let $G$ be a connected graph of order $n$ with Wiener index $W$. Then $$\partial_1\geq\frac{2W}{n}$$
and the equality holds if and only if $G$ is a transmission regular.
\end{lemma}

Let $G(\lfloor\frac{d+1}{2}\rfloor,\lceil\frac{d+1}{2}\rceil)$ denote by the graph obtained from $K_{n-1-d}$ by joining an
endvertex of $P_{\lfloor\frac{d+1}{2}\rfloor}$ and $P_{\lceil\frac{d+1}{2}\rceil}$ to each vertex of $K_{n-1-d}$.
Zhang \cite{XL} determined that $G(\lfloor\frac{d+1}{2}\rfloor,\lceil\frac{d+1}{2}\rceil)$ attains the minimal distance spectral
radius among all graphs with given diameter $d$ and order $n$. Then we have the following result.

\begin{lemma}\label{lem3.3}
Let $G$ be a connected graph of order $n$ with diameter $d\geq 3$. Then $$\partial_1> n-2+d.$$
\end{lemma}

\begin{proof}
Note that $G(\lfloor\frac{d+1}{2}\rfloor,\lceil\frac{d+1}{2}\rceil)$ attains the minimal distance spectral
radius among all graphs with given diameter $d$ and order $n$. Then we only need to show that
$\partial_1(G(\lfloor\frac{d+1}{2}\rfloor,\lceil\frac{d+1}{2}\rceil))>n-2+d$ when $d\geq 3.$ By a simple calculation, we have
$$W= \left\{
  \begin{array}{ccccccccccc}
       {d+2\choose{3}}+{n-d-1\choose{2}}+\frac{d^2+2d}{4}(n-d-1), \ \mbox{$n$ is even,} \\[4mm]
       {d+2\choose{3}}+{n-d-1\choose{2}}+\frac{d^2+2d+5}{4}(n-d-1), \ \mbox{$n$ is odd.}
\end{array}
\right.$$ By Lemma \ref{lem3.2} and $G$ is not transmission regular, we only need to show that $2W\geq n(n-1-d)$.
For the sake of simplicity, we only give the proof of $n$ is even.
\begin{eqnarray*}
2W&=&\frac{d(d+1)(d+2)}{3}+(n-d-1)(n-d-2)+\frac{d^2+2d}{2}(n-d-1)\\
&=&\frac{d(d+1)(d+2)}{3}+(n-d-1)(n-d-2)+\frac{d^2+2d}{2}(n-d-1)\\
&=&\frac{d(d+1)(d+2)}{3}+n^2+(\frac{d^2}{2}-d-1)n-\frac{d^2}{2}(d+1)\\
&\geq&n(n-2+d) \ \mbox{since $d\geq 4$}.
\end{eqnarray*}
This completes the proof.
\end{proof}

Aouchiche and Hansen \cite{AH} showed that $$\partial_1-\rho\geq n-2,$$ with equality holding if and only if $G\cong K_n.$
In the following, we will give the lower bound on $\partial_1-\rho$ when $G\ncong K_n.$

\begin{theorem}
Let $G\ncong K_n$ be a connected graph of order $n$ with remoteness $\rho$. Then $$\partial_1-\rho\geq \frac{n-1+\sqrt{(n-1)^2+8}}{2}-\frac{n}{n-1},$$
and the equality holds if and only if $G\cong K_n-e$ where $e$ is an edge of $G$.
\end{theorem}
\begin{proof}
Note that $\partial_1(K_n-e)-\rho(K_n-e)=\frac{n-1+\sqrt{(n-1)^2+8}}{2}-\frac{n}{n-1}$. It is easy to see that
$\sqrt{(n-1)^2+8}<n-1+\frac{4}{n-1}.$ Then $$\partial_1(K_n-e)-\rho(K_n-e)=\frac{n-1+\sqrt{(n-1)^2+8}}{2}-\frac{n}{n-1}<n-2+\frac{1}{n-1}.$$
Let $G$ be a connected graph with diameter $d$. If $d\geq 3$, then by Lemmas \ref{lem3.1} and \ref{lem3.3},
we have
\begin{eqnarray*}
\partial_1(G)-\rho(G)&\geq& n-2+\frac{d^2-d}{2(n-1)}\\[2mm]
&>&n-2+\frac{1}{n-1}\\[2mm]
&>&\partial_1(K_n-e)-\rho(K_n-e).
\end{eqnarray*}
So in the following, we may assume that $d=2$. Then
$$\rho=\frac{\delta+2(n-1-\delta)}{n-1}=\frac{2(n-1)-\delta}{n-1}=2-\frac{\delta}{n-1}$$
and by Lemma \ref{lem3.1}, we have
$$\partial_1\geq\frac{n(n-1)+2(n-1-\delta)}{n}=n+1-\frac{2(1+\delta)}{n}.$$
It follows that
\begin{eqnarray*}
\partial_1-\rho&\geq& n-1-\frac{2(1+\delta)}{n}+\frac{\delta}{n-1}\\[2mm]
&=&n-1-\frac{n\delta+(n-1)(2\delta+2)}{n(n-1)}.
\end{eqnarray*}
If $\delta\leq n-3$, then we have $$\partial_1-\rho\geq n-2+\frac{1}{n-1}>\partial_1(K_n-e)-\rho(K_n-e).$$
If $\delta=n-2$ and note that $\partial_1(G)>\partial_1(G+e)$, then $\partial_1(G)-\rho(G)\geq\partial_1(K_n-e)-\rho(K_n-e)$
with equality holding if and only if $G\cong K_n-e$. Thus we complete the proof.
\end{proof}

\noindent {\textbf{Acknowledgement.}} The first author is supported
by the National Natural Science Foundation of China (No. 11401211
and 11471211), the China Postdoctoral Science Foundation (No.
2014M560303) and Fundamental Research Funds for the Central
Universities (No. 222201414021). The second author is supported by
the National Research Foundation funded by the Korean government
with Grant no. 2013R1A1A2009341. The third author is supported by
the National Natural Science Foundation of China (No. 11161046) and
by Xingjiang Talent Youth Project (No. 2013721012).


\begin{thebibliography}{13}

\bibitem{AH} M. Aouchiche, P. Hansen, Proximity, remoteness and distance eigenvalues of a graph, https://www.gerad.ca/en/papers/G-2014-68.

\bibitem{Indulal}G. Indulal, Sharp bounds on the distance spectral radius and the distance energy of graphs, Linear Algebra Appl. 430 (2009) 106--113.
\bibitem{Lin-1} H.Q. Lin, On the least distance eigenvalue and its applications on the
distance spread, Discrete Mathematics 338 (2015) 868--874.

\bibitem{Lin-2} H.Q. Lin, Y. Hong, J.F. Wang, J.L. Shu, On the distance spectrum of graphs,
Linear Algebra and its Applications 439 (2013) 1662--1669.

\bibitem{Merris} R. Merris, The distance spectrum of a tree. J. Graph Theory 14 (1990)
365--369.

\bibitem{XL}X.L. Zhang, On the distance spectral radius of some graphs,
Linear Algebra Appl. 437 (2012) 1930--1941.





\end{thebibliography}
 \end{document}